\documentclass[12pt]{article}%
\usepackage{amssymb}
\usepackage{amsmath}
\usepackage{amsfonts}
\usepackage{mitpress}
\usepackage{graphicx}%
\providecommand{\U}[1]{\protect\rule{.1in}{.1in}}
\newtheorem{theorem}{Theorem}

\newtheorem{definition}[theorem]{Definition}
\newtheorem{example}[theorem]{Example}

\newtheorem{lemma}[theorem]{Lemma}
\newtheorem{notation}[theorem]{Notation}

\newtheorem{remark}[theorem]{Remark}

\newenvironment{proof}[1][Proof]{\noindent\textbf{#1.} }{\ \rule{0.5em}{0.5em}}
\newdimen\dummy
\dummy=\oddsidemargin
\begin{document}

\title{Binary signals: a note on the prime period of a point}
\author{Serban E. Vlad\\Str. Zimbrului, Nr. 3, Bl. PB68, Ap. 11, 410430, Oradea, email: serban\_e\_vlad@yahoo.com}
\maketitle

\begin{abstract}
The 'nice' $x:\mathbf{R}\rightarrow\{0,1\}^{n}$ functions from the
asynchronous systems theory are called signals. The periodicity of a point of
the orbit of the signal $x$ is defined and we give a note on the existence of
the prime period.


\end{abstract}


The asynchronous systems are the models of the digital electrical circuits and
the 'nice' functions representing their inputs and states are called signals.
Such systems are generated by Boolean functions that iterate like the
dynamical systems, but the iterations happen on some coordinates only, not on
all the coordinates (unlike the dynamical systems). In order to study their
periodicity, we need to study the periodicity of the (values of the) signals
first. Our present aim is to define and to characterize the periodicity and
the prime period of a point of the orbit of a signal.

\begin{definition}
\label{Not1}The set $\mathbf{B}=\{0,1\}$ is a field relative to $^{\prime
}\oplus^{\prime},$ $^{\prime}\cdot^{\prime},$ the modulo 2 sum and the
product$.$ A linear space structure is induced on $\mathbf{B}^{n},n\geq1.$
\end{definition}

\begin{notation}
\label{Not4}$\chi_{A}:\mathbf{R}\rightarrow\mathbf{B}$ is the notation of the
characteristic function of the set $A\subset\mathbf{R}:\forall t\in
\mathbf{R},$%
\[
\chi_{A}(t)=\left\{
\begin{array}
[c]{c}%
1,if\;t\in A,\\
0,otherwise
\end{array}
\right.  .
\]

\end{notation}

\begin{definition}
\label{Def2}The \textbf{continuous time signals} are the functions
$x:\mathbf{R}\rightarrow\mathbf{B}^{n}$ of the form $\forall t\in\mathbf{R},$%
\begin{equation}
x(t)=\mu\cdot\chi_{(-\infty,t_{0})}(t)\oplus x(t_{0})\cdot\chi_{\lbrack
t_{0},t_{1})}(t)\oplus...\oplus x(t_{k})\cdot\chi_{\lbrack t_{k},t_{k+1}%
)}(t)\oplus... \label{per189}%
\end{equation}
where $\mu\in\mathbf{B}^{n}$ and $t_{k}\in\mathbf{R},k\in\mathbf{N}$ is
strictly increasing and unbounded from above$.$ Their set is denoted by
$S^{(n)}.$ $\mu$ is usually denoted by $x(-\infty+0)$ and is called the
\textbf{initial value} of $x$.
\end{definition}

\begin{definition}
The \textbf{left limit} $x(t-0)$ of $x(t)$ from (\ref{per189}) is by
definition the function $\forall t\in\mathbf{R},$%
\begin{equation}
x(t-0)=\mu\cdot\chi_{(-\infty,t_{0}]}(t)\oplus x(t_{0})\cdot\chi_{(t_{0}%
,t_{1}]}(t)\oplus...\oplus x(t_{k})\cdot\chi_{(t_{k},t_{k+1}]}(t)\oplus...
\end{equation}

\end{definition}

\begin{remark}
The definition of $x(t-0)$ does not depend on the choice of $(t_{k})$ that is
not unique in (\ref{per189}); for any $t^{\prime}\in\mathbf{R},$ the existence
of $x(t^{\prime}-0)$ is used in applications under the form $\exists
\varepsilon>0,\forall\xi\in(t^{\prime}-\varepsilon,t^{\prime}),x(\xi
)=x(t^{\prime}-0).$
\end{remark}

\begin{definition}
The set $Or(x)=\{x(t)|t\in\mathbf{R}\}$ is called the \textbf{orbit} of $x$.
\end{definition}

\begin{notation}
\label{Not5}For $x\in S^{(n)}$ and $\mu\in Or(x),$ we denote
\begin{equation}
\mathbf{T}_{\mu}^{x}=\{t|t\in\mathbf{R},x(t)=\mu\}.
\end{equation}

\end{notation}

\begin{definition}
The point $\mu\in Or(x)$ is called a \textbf{periodic point} of $x\in
S^{(n)}\,$or of $Or(x)$ if $T>0,t^{\prime}\in\mathbf{R}$ exist such that%
\begin{equation}
(-\infty,t^{\prime}]\subset\mathbf{T}_{x(-\infty+0)}^{x},\label{per492}%
\end{equation}%
\begin{equation}
\forall t\in\mathbf{T}_{\mu}^{x}\cap\lbrack t^{\prime},\infty),\{t+zT|z\in
\mathbf{Z}\}\cap\lbrack t^{\prime},\infty)\subset\mathbf{T}_{\mu}%
^{x}.\label{per483}%
\end{equation}
In this case $T$ is called the \textbf{period} of $\mu$ and the least $T$ like
above is called the \textbf{prime period} of $\mu$.
\end{definition}

\begin{theorem}
\label{The75}Let $x\in S^{(n)},$ $\mu=x(-\infty+0),$ $T>0$ and the points
$t_{0},t_{1}\in\mathbf{R}$ having the property that%
\begin{equation}
t_{0}<t_{1}<t_{0}+T,
\end{equation}%
\begin{equation}
(-\infty,t_{0})\cup\lbrack t_{1},t_{0}+T)\cup\lbrack t_{1}+T,t_{0}%
+2T)\cup\lbrack t_{1}+2T,t_{0}+3T)\cup...=\mathbf{T}_{\mu}^{x}\label{per501}%
\end{equation}
hold$.$

a) For any $t^{\prime}\in\lbrack t_{1}-T,t_{0}),$ the properties
(\ref{per492}), (\ref{per483}) are fulfilled and for any $t^{\prime}%
\notin\lbrack t_{1}-T,t_{0}),$ at least one of the properties (\ref{per492}),
(\ref{per483}) is false.

b) For any $T^{\prime}>0,t^{\prime\prime}\in\mathbf{R}$ such that%
\begin{equation}
(-\infty,t^{\prime\prime}]\subset\mathbf{T}_{x(-\infty+0)}^{x},\label{per503}%
\end{equation}%
\begin{equation}
\forall t\in\mathbf{T}_{\mu}^{x}\cap\lbrack t^{\prime\prime},\infty
),\{t+zT^{\prime}|z\in\mathbf{Z}\}\cap\lbrack t^{\prime\prime},\infty
)\subset\mathbf{T}_{\mu}^{x},\label{per504}%
\end{equation}
we have $T^{\prime}\geq T$ and $t^{\prime\prime}\in\lbrack t_{1}-T^{\prime
},t_{0}).$
\end{theorem}

\begin{proof}
a) Let $t^{\prime}\in\lbrack t_{1}-T,t_{0}).$ From%
\[
(-\infty,t^{\prime}]\subset(-\infty,t_{0})\subset\mathbf{T}_{\mu}^{x},
\]
we infer the truth of (\ref{per492}).

Furthermore, we have%
\[
\mathbf{T}_{\mu}^{x}\cap\lbrack t^{\prime},\infty)=[t^{\prime},t_{0}%
)\cup\lbrack t_{1},t_{0}+T)\cup\lbrack t_{1}+T,t_{0}+2T)\cup\lbrack
t_{1}+2T,t_{0}+3T)\cup...
\]
and we take an arbitrary $t\in\mathbf{T}_{\mu}^{x}\cap\lbrack t^{\prime
},\infty).$ If $t\in\lbrack t^{\prime},t_{0}),$ then
\[
\{t+zT|z\in\mathbf{Z}\}\cap\lbrack t^{\prime},\infty
)=\{t,t+T,t+2T,...\}\subset
\]%
\[
\subset\lbrack t^{\prime},t_{0})\cup\lbrack t^{\prime}+T,t_{0}+T)\cup\lbrack
t^{\prime}+2T,t_{0}+2T)\cup...\subset
\]%
\[
\subset\lbrack t^{\prime},t_{0})\cup\lbrack t_{1},t_{0}+T)\cup\lbrack
t_{1}+T,t_{0}+2T)\cup...\subset\mathbf{T}_{\mu}^{x}.
\]
If $\exists k_{1}\geq0,t\in\lbrack t_{1}+k_{1}T,t_{0}+(k_{1}+1)T),$ then there
are two possibilities:

Case $t\in\lbrack t^{\prime}+(k_{1}+1)T,t_{0}+(k_{1}+1)T),$ when
\[
\{t+zT|z\in\mathbf{Z}\}\cap\lbrack t^{\prime},\infty)=\{t+(-k_{1}%
-1)T,t+(-k_{1})T,t+(-k_{1}+1)T,...\}\subset
\]%
\[
\subset\lbrack t^{\prime},t_{0})\cup\lbrack t^{\prime}+T,t_{0}+T)\cup\lbrack
t^{\prime}+2T,t_{0}+2T)\cup...\subset
\]%
\[
\subset\lbrack t^{\prime},t_{0})\cup\lbrack t_{1},t_{0}+T)\cup\lbrack
t_{1}+T,t_{0}+2T)\cup...\subset\mathbf{T}_{\mu}^{x};
\]

Case $t\in\lbrack t_{1}+k_{1}T,t^{\prime}+(k_{1}+1)T),$ when%
\[
\{t+zT|z\in\mathbf{Z}\}\cap\lbrack t^{\prime},\infty)=\{t+(-k_{1}%
)T,t+(-k_{1}+1)T,t+(-k_{1}+2)T,...\}\subset
\]%
\[
\subset\lbrack t_{1},t^{\prime}+T)\cup\lbrack t_{1}+T,t^{\prime}%
+2T)\cup\lbrack t_{1}+2T,t^{\prime}+3T)\cup...\subset
\]%
\[
\subset\lbrack t_{1},t_{0}+T)\cup\lbrack t_{1}+T,t_{0}+2T)\cup\lbrack
t_{1}+2T,t_{0}+3T)\cup...\subset\mathbf{T}_{\mu}^{x}.
\]

We suppose now that $t^{\prime}\notin\lbrack t_{1}-T,t_{0}).$ If $t^{\prime
}<t_{1}-T,$ we notice that $\max\{t^{\prime},t_{0}-T\}<t_{1}-T$ and that for
any $t\in\lbrack\max\{t^{\prime},t_{0}-T\},t_{1}-T),$ we have $t\in
\mathbf{T}_{\mu}^{x}\cap\lbrack t^{\prime},\infty)$ but%
\[
t+T\in\{t+zT|z\in\mathbf{Z}\}\cap\lbrack t^{\prime},\infty)\cap\lbrack
t_{0},t_{1})\mathbf{,}%
\]
thus $t+T\notin\mathbf{T}_{\mu}^{x}$ and (\ref{per483}) is false. On the other
hand if $t^{\prime}\geq t_{0},$ then $x(t_{0})\neq\mu$ implies $t_{0}%
\notin\mathbf{T}_{\mu}^{x}$ and consequently (\ref{per492}) is false.

b) The fact that $t^{\prime\prime}\in\lbrack t_{1}-T^{\prime},t_{0})$ is
proved similarly with the statement $t^{\prime}\in\lbrack t_{1}-T,t_{0})$ from
a): $t^{\prime\prime}\geq t_{0}$ is in contradiction with (\ref{per503}) and
$t^{\prime\prime}<t_{1}-T^{\prime}$ is in contradiction with (\ref{per504}).

We suppose now against all reason that (\ref{per503}), (\ref{per504}) are true
and $T^{\prime}<T.$ Let us note in the beginning that%
\[
\max\{t_{1},t_{0}+T-T^{\prime}\}<\min\{t_{0}+T,t_{1}+T-T^{\prime}\}
\]
is true, since all of $t_{1}<t_{0}+T,t_{1}<t_{1}+T-T^{\prime},t_{0}%
+T-T^{\prime}<t_{0}+T,t_{0}+T-T^{\prime}<t_{1}+T-T^{\prime}$ are true. We
infer that any $t\in\lbrack\max\{t_{1},t_{0}+T-T^{\prime}\},\min
\{t_{0}+T,t_{1}+T-T^{\prime}\})$ fulfills $t\in\lbrack t_{1},t_{0}%
+T)\subset\mathbf{T}_{\mu}^{x}\cap\lbrack t^{\prime\prime},\infty)$ and%
\[
t_{0}+T\leq\max\{t_{1}+T^{\prime},t_{0}+T\}\leq t+T^{\prime}<\min
\{t_{0}+T+T^{\prime},t_{1}+T\}\leq t_{1}+T
\]
in other words $t+T^{\prime}\in\{t+zT^{\prime}|z\in\mathbf{Z}\}\cap\lbrack
t^{\prime\prime},\infty),$ but $t+T^{\prime}\in\lbrack t_{0}+T,t_{1}+T),$ thus
$t+T^{\prime}\notin\mathbf{T}_{\mu}^{x},$ contradiction with (\ref{per504}).
We conclude that $T^{\prime}\geq T.$
\end{proof}

\begin{lemma}
\label{Lem8}We suppose that the point $\mu\in Or(x)$ is periodic:
$T>0,t^{\prime}\in\mathbf{R}$ exist such that (\ref{per492}), (\ref{per483})
hold. If for $t_{1}<t_{2}$ we have $[t_{1},t_{2})\subset\mathbf{T}_{\mu}%
^{x}\cap\lbrack t^{\prime},\infty),$ then $\forall k\geq1,[t_{1}%
+kT,t_{2}+kT)\subset\mathbf{T}_{\mu}^{x}.$
\end{lemma}

\begin{proof}
Let $k\geq1$ and $t\in\lbrack t_{1}+kT,t_{2}+kT)$ be arbitrary. As
$t-kT\in\lbrack t_{1},t_{2})$ and from the hypothesis $t-kT\in\mathbf{T}_{\mu
}^{x}\cap\lbrack t^{\prime},\infty),$ we have from (\ref{per483}) that%
\[
t\in\{t-kT+zT|z\in\mathbf{Z}\}\cap\lbrack t^{\prime},\infty)\subset
\mathbf{T}_{\mu}^{x}.
\]

\end{proof}

\begin{theorem}
\label{The76}We ask that $x$ is not constant and let the point $\mu
=x(-\infty+0)$ be given, as well as $T>0,t^{\prime}\in\mathbf{R}$ such that
(\ref{per492}), (\ref{per483}) hold. We define $t_{0},t_{1}\in\mathbf{R}$ by
the requests%
\begin{equation}
\forall t<t_{0},x(t)=\mu,\label{per875}%
\end{equation}%
\begin{equation}
x(t_{0})\neq\mu,\label{per876}%
\end{equation}%
\begin{equation}
t_{1}<t_{0}+T,\label{per891}%
\end{equation}%
\begin{equation}
\forall t\in\lbrack t_{1},t_{0}+T),x(t)=x(t_{0}+T-0),\label{per877}%
\end{equation}%
\begin{equation}
x(t_{1}-0)\neq x(t_{1}).\label{per878}%
\end{equation}
Then the following statements are true:%
\begin{equation}
t_{1}-T\leq t^{\prime}<t_{0}<t_{1},\label{per879}%
\end{equation}%
\begin{equation}
(-\infty,t_{0})\cup\lbrack t_{1},t_{0}+T)\cup\lbrack t_{1}+T,t_{0}%
+2T)\cup\lbrack t_{1}+2T,t_{0}+3T)\cup...\subset\mathbf{T}_{\mu}%
^{x}.\label{per880}%
\end{equation}

\end{theorem}

\begin{proof}
The fact that $x$ is not constant assures the existence of $t_{0}$ as defined
by (\ref{per875}), (\ref{per876}). On the other hand $t_{1}$ as defined by
(\ref{per891}), (\ref{per877}), (\ref{per878}) exists itself, since if,
against all reason, we would have%
\begin{equation}
\forall t<t_{0}+T,x(t)=x(t_{0}+T-0), \label{per884}%
\end{equation}
then (\ref{per875}), (\ref{per876}), (\ref{per884}) would be contradictory. By
the comparison between (\ref{per875}), (\ref{per876}), (\ref{per877}),
(\ref{per878}) we infer $t_{0}\leq t_{1}.$ From (\ref{per492}), (\ref{per875}%
), (\ref{per876}) we get $t^{\prime}<t_{0}.$

Case $t^{\prime}\in\lbrack t_{1}-T,t_{0})$

In this situation%
\begin{equation}
t^{\prime}+T\in\lbrack t_{1},t_{0}+T), \label{per885}%
\end{equation}%
\begin{equation}
\mu\overset{(\ref{per492})}{=}x(t^{\prime})\overset{(\ref{per483})}%
{=}x(t^{\prime}+T)\overset{(\ref{per877}),(\ref{per885})}{=}x(t_{1})
\label{per886}%
\end{equation}
and from (\ref{per875}), (\ref{per876}), $t_{0}\leq t_{1},$ (\ref{per886}) we
have that $t_{0}<t_{1}.$ (\ref{per879}) is true.

Let us note that (\ref{per877}), $t^{\prime}<t_{1},$ (\ref{per886}) imply
$[t_{1},t_{0}+T)\subset\mathbf{T}_{\mu}^{x}\cap\lbrack t^{\prime},\infty) $
and, from Lemma \ref{Lem8} together with (\ref{per875}), (\ref{per876}) we
obtain the truth of (\ref{per880}).

Case $t^{\prime}<t_{1}-T$

As $t_{1}-T<t_{0}$ we can write%
\begin{equation}
\mu\overset{(\ref{per875})}{=}x(t_{1}-T)\overset{(\ref{per483})}{=}x(t_{1})
\label{per887}%
\end{equation}
and the property of existence of the left limit of $x$ in $t_{1}$ shows the
existence of $\varepsilon>0$ with%
\begin{equation}
\forall t\in(t_{1}-\varepsilon,t_{1}),x(t)=x(t_{1}-0). \label{per888}%
\end{equation}
We take $\varepsilon^{\prime}\in(0,\min\{t_{1}-T-t^{\prime},\varepsilon\}),$
thus for any $t\in(t_{1}-T-\varepsilon^{\prime},t_{1}-T)$ we have%
\begin{equation}
t+T\in(t_{1}-\varepsilon^{\prime},t_{1})\subset(t_{1}-\varepsilon,t_{1}),
\label{per889}%
\end{equation}
and since%
\begin{equation}
t>t_{1}-T-\varepsilon^{\prime}>t^{\prime}, \label{per892}%
\end{equation}%
\begin{equation}
t<t_{1}-T<t_{0} \label{per893}%
\end{equation}
we conclude%
\begin{equation}
\mu\overset{(\ref{per875}),(\ref{per893})}{=}x(t)\overset{(\ref{per483}%
),(\ref{per892})}{=}x(t+T)\overset{(\ref{per888}),(\ref{per889})}{=}%
x(t_{1}-0). \label{per890}%
\end{equation}
Equations (\ref{per878}), (\ref{per887}), (\ref{per890}) are contradictory,
thus $t^{\prime}<t_{1}-T$ is impossible.
\end{proof}

\begin{example}
We take $x\in S^{(1)},$%
\[
x(t)=\chi_{(-\infty,0)}(t)\oplus\chi_{\lbrack1,2)}\oplus\chi_{\lbrack
3,5)}\oplus\chi_{\lbrack6,7)}\oplus\chi_{\lbrack8,10)}\oplus\chi
_{\lbrack11,12)}\oplus...
\]
In this example $\mu=1,t_{0}=0,t_{1}=3,T=5$ is prime period and $t^{\prime}%
\in\lbrack-2,0).$ We note that $T$ may be prime period without equality at
(\ref{per880}) but if we have equality at (\ref{per880}) then, from Theorem
\ref{The75}, $T$ is prime period.
\end{example}

\begin{remark}
Theorems \ref{The75} and \ref{The76} refer to the case when the periodic point
$\mu$ coincides with $x(-\infty+0).$ The situation when $\mu\neq x(-\infty+0)$
is not different in principle.
\end{remark}

\end{document}